\documentclass{amsart}

\usepackage{amssymb, amsfonts,amsthm,amsmath,latexsym}
\usepackage[all]{xy}

\def\bt{\begin{thm}}
\def\et{\end{thm}}
\def\bl{\begin{lem}}
\def\el{\end{lem}}
\def\bd{\begin{defi}}
\def\ed{\end{defi}}
\def\bc{\begin{cor}}
\def\ec{\end{cor}}
\def\bp{\begin{proof}}
\def\ep{\end{proof}}
\def\br{\begin{rem}}
\def\er{\end{rem}}

\newtheorem{thm}{Theorem}[section]
\newtheorem{prop}[thm]{Proposition}
\newtheorem{lem}[thm]{Lemma}
\newtheorem{defn}[thm]{Definition}

\newtheorem{rem}[thm]{Remark}
\newtheorem{cor}[thm]{Corollary}

\numberwithin{equation}{section}

\newcommand{\cohomology}{H^{1,1}(X,\mathbb{R})}

\newcommand{\nef}{H^{1,1}_{nef}(X,\mathbb{R})}

\newcommand{\pk}{\Bbb{P}^k}

\newcommand{\bthm}{\begin{thm}}
\newcommand{\ethm}{\end{thm}}
\newcommand{\bstp}{\begin{stp}}
\newcommand{\estp}{\end{stp}}
\newcommand{\blemma}{\begin{lemma}}
\newcommand{\elemma}{\end{lemma}}
\newcommand{\bprop}{\begin{prop}}
\newcommand{\eprop}{\end{prop}}
\newcommand{\bpf}{\begin{pf}}
\newcommand{\epf}{\end{pf}}
\newcommand{\bdefn}{\begin{defn}}
\newcommand{\edefn}{\end{defn}}
\newcommand{\brk}{\begin{rmrk}}
\newcommand{\erk}{\end{rmrk}}
\newcommand{\bcrl}{\begin{crl}}
\newcommand{\ecrl}{\end{crl}}

\title{On the automorphism group of rational manifolds}
\author{Turgay Bayraktar}
\date{\today}
\address{Mathematics Department, Johns Hopkins University 21218 Maryland, USA}
\email{bayraktar@jhu.edu}
\keywords{automorphism, rational manifold, topological entropy}
\subjclass[2000]{37F10, 32H50, 14E09}
\begin{document}

\begin{abstract}
 In this note, we prove that every automorphism of a rational manifold which is obtained from $\pk$ by a finite sequence blow-ups along smooth centers of dimension at most $r$ with $k>2r+2$ has zero topological entropy. 
\end{abstract}

\maketitle

\section{Introduction}

A holomorphic automorphism of a compact K\"ahler manifold has positive topological entropy if and of if absolute value of one of the eigenvalues of $f^*$ on the cohomology $H^*(X,\Bbb{C})$ is larger than one. It follows from the results of Cantat \cite{C1,C2} that a compact complex surface admit an automorphism with positive entropy if it is K\"ahler and bimeromorphic to one of the following: a rational surface, a torus,  a $K3$ surface or an Enriques surface.  In particular, if $X$ is a rational surface admitting an automorphism with positive entropy then $X$ is obtained from $\Bbb{P}^2$ by blowing up a finite sequence of at least ten points \cite{Nagata}. Examples of rational surface automorphisms with positive entropy were given by \cite{BK2,BKauto,McMullen}. On the other hand, in higher dimensions the question if one can obtain automorphisms with interesting dynamics by blowing up certain subvarieties of $\pk$ remained open. Our first result partially addresses this question:
 
\begin{thm}\label{rational}
Let $X$ be a rational manifold such that $X=X_m$ and $\pi_i:X_{i+1}\to X_i$ is obtained by blowing up a smooth irreducible subvariety of dimension at most $r$ with $k>2r+2$ in $X_i$ where $X_0=\pk.$  If $f:X\to X$ is a holomorphic automorphism then $f$ has zero topological entropy.
\end{thm}
In particular, if $X$ is obtained from $\pk$ with $k\geq 3$ by blowing up  a finite sequence of points then every holomorphic automorphism of $X$ has zero topological entropy. This was observed in \cite{tuyen} when $k=3$. \\ \indent 
A cohomology class $\alpha \in \cohomology$ is called \textit{numerically effective} (nef in short) if $\alpha$ lies in the closure of classes of K\"ahler forms. Following \cite{Kawamata} we define \textit{numerical dimension} of a nef class as
  $$\nu(\alpha):=\max\{p\in\Bbb{N}: \alpha^p:=\alpha\wedge \dots \wedge \alpha\neq 0\ \text{in} \ H^{p,p}(X,\Bbb{R})\}.$$ 
Next, we prove that if an automorphism of a compact K\"ahler manifold preserves a numerically-effective class $\alpha \in \cohomology$ with large numerical dimension then it has zero entropy.
\begin{thm}\label{numerical}
Let $X$ be a compact K\"ahler manifold with and $f\in Aut(X).$ If there exists a nef class $\alpha\in \cohomology$ such that $\nu(\alpha)\geq \dim X-1$ and $f^*\alpha= \alpha$ then $f$ has zero topological entropy. 
\end{thm}
 In some sense Theorem \ref{numerical} can be considered a generalization of Liberman's result \cite{Lieberman} (see also \cite{Zhang} for big and nef case) which asserts that if $X$ is a compact K\"ahler manifold and $f\in Aut(X)$ preserves a K\"ahler class then an iterate of $f$ belongs to $Aut_0(X),$ the connected component of the identity, and hence $f$ has zero topological entropy.

\section*{Acknowledgement}
I would like to thank Mattias Jonsson and Tuyen Truong for their valuable comments on an earlier draft. I am also grateful to Brian Lehmann for stimulating correspondence.

\section{Preliminaries}
Let $X$ be a compact K\"ahler manifold. We denote the de Rham (respectively Dolbeault) cohomology groups by $H^{2p}(X,\Bbb{R})$ (respectively $H^{p,p}(X,\Bbb{C}))$ and define $$H^{p,p}(X,\Bbb{R}):=H^{p,p}(X,\Bbb{C})\cap H^{2p}(X,\Bbb{R}).$$ Note that $H^{2p}(X,\Bbb{R}),\ H^{p,p}(X,\Bbb{C})$ are finite dimensional and one can identify $H^{p,p}(X,\Bbb{R})$ with a real subspace of $H^{p,p}(X,\Bbb{C}).$ In the sequel, we implicitly use the fact that the cohomology classes can be defined in terms of smooth forms or currents. We refer the reader to \cite{GH} for basic results in Hodge theory. A cohomology class $\alpha \in \cohomology$ is called \textit{numerically effective} (nef in short) if $\alpha$ lies in the closure of classes of K\"ahler forms. 
The set of nef classes $\nef$ forms a closed convex cone which is strict that is $\nef \cap -\nef=\{0\}.$

 We let $Pic(X)$ denote the Picard gourp of $X$ that is isomorphism classes of line bundles with the group operation tensor product and denote the Chern map by $$c_1:Pic(X) \to H^2(X,\Bbb{Z}).$$ By a slight abuse of notation we will write $c_1(L)\in H^2(X,\Bbb{R})$ where we consider the image of $c_1(L)$ under the inclusion $i:H^2(X,\Bbb{Z})\to H^2(X,\Bbb{R}).$ The Neron-Severi group of $X$ is defined by $NS(X)=c_1(Pic(X))\subset H^2(X,\Bbb{R})$ that is the Chern classes of line bundles on $X.$ It follows from Lefschetz theorem on $(1,1)$ classes that $$NS(X)=H^2(X,\Bbb{Z})\cap H^{1,1}(X,\Bbb{R}).$$ We also let $NS_{\Bbb{R}}(X)$ be the real vector space $NS_{\Bbb{R}}(X)=NS(X) \otimes \Bbb{R}\subset H^2(X,\Bbb{R}).$\\ \indent
  A holomorphic line bundle $L$ is called \textit{numerically effective} (nef) if $$L\cdot C=\int_Cc_1(L)\geq 0$$ for every curve $C\subset X$. It follows from \cite{DemS} that $L$ is nef if and only if $c_1(L)\in \nef$. A line bundle $L$ is said to be \textit{big} if $\kappa(L)=\dim X$ where $\kappa(L)$ denotes the Kodaira-Iitaka dimension of $L.$ It is well-known that a nef line bundle $L$ is big if and only if $L^k:=\int_Xc_1(L)^k>0.$\\ \indent

\subsection{Dynamics of automorphisms of compact K\"ahler manifolds}
Let $X$ be a compact K\"ahler manifold of dimension $k$ and $\omega$ be a fixed K\"ahler form on $X$. We let $Aut(X)$ denote the set of holomorphic automorphisms of $X$. Every $f\in Aut(X)$ induces a linear action 
$$f^*:H^{p,p}(X,\Bbb{R}) \to H^{p,p}(X,\Bbb{R})$$
$$f^*\{\theta\}:=\{f^*\theta\}$$
where $\{\theta\}$ denotes the class of the smooth $(p,p)$ form $\theta$ in $H^{p,p}(X,\Bbb{R}).$\\ \indent
For $f\in Aut(X)$ the $i^{th}$ \textit{dynamical degree} of $f$ is defined by 
$$\lambda_i(f):=\limsup_{m\to \infty}(\int_X(f^n)^*\omega^i\wedge \omega^{k-i} )^{\frac{1}{n}}.$$
Since $X$ is compact this definition is independent of $\omega.$
The following properties of dynamical degrees are well known \cite{DS04}:
\begin{lem} \label{degree}
Let $f:X\to X$ be an automorphism. Then 

\begin{itemize}
\item[(i)] $1\leq \lambda_i$ is the spectral radius of $f^*_{|_{H^{i,i}(X,\Bbb{R})}}.$  

\item[(ii)] $i \to \log\lambda_i(f)$ is concave on $\{0,1,\dots, k\}.$
\item[(iii)] $\lambda_1(f)^i \geq \lambda_i(f)$ and $\lambda_i(f)^i\geq \lambda_1(f)$ for $1\leq i\leq k-1$.

\item[(iv)] $\lambda_i(f)=\lambda_{k-i}(f^{-1})$ for $i\in\{0,1,\dots,k\}.$
\end{itemize}

\end{lem}

 
\begin{thm}[Gromov and Yomdin]\label{GY}
Let $f\in Aut(X)$ then the topological entropy of $f$ is given by
$$h_{top}(f)=\max_{0\leq i \leq k}\log\lambda_i(f).$$ 
\end{thm}

 The next proposition will be useful in the proof of Theorem \ref{numerical}:
 \begin{prop}\cite{DSauto}\label{DS}
 Let $\alpha,\alpha',\alpha_1,\dots, \alpha_r \in \cohomology$ be nef classes where $r\leq k-2.$ 
 \begin{itemize}
 \item[(1)] If $\alpha\wedge \alpha'=0$ then $\alpha$ and $\alpha'$ are colinear. 
 \item[(2)] If $\alpha\wedge \alpha' \wedge \alpha_1 \dots \wedge \alpha_r=0\ \text{in}\ H^{r+2,r+2}(X,\Bbb{R})$ then there exists real numbers $(a,b)\not=(0,0)$ such that $$(a\alpha+b\alpha')\wedge \alpha_1\dots \wedge \alpha_r=0.$$ Furthermore, if $\alpha'\wedge \alpha_1\dots\wedge \alpha_r\not=0$ then the pair $(a,b)$ is unique up to a multiplicative constant. 
 \end{itemize}
 \end{prop}
\section{Proofs of Theorem 1.1 and Theorem 1.2}
\begin{proof}[Proof of Theorem \ref{rational}]
Let $X=X_m$ be a rational manifold where $\pi_i:X_{i+1}\to X_i$ is obtained by blowing up a smooth irreducible subvariety $Y_i \subset X_i$ of dimension at most $r$ and $X_0=\pk$ with $k> 2r+2.$\\ By Lemma \ref{degree} we have $\lambda_1(f)^{k-r-1}\geq \lambda_{k-r-1}(f).$ 
First, we will show that $$\lambda_1^{k-r-1}(f)=\lambda_{k-r-1}(f).$$
Indeed, assuming otherwise $\lambda_1^{k-r-1}(f)>\lambda_{k-r-1}(f)$ we will derive a contradiction. Since $f^*$ preserves the nef cone it follows from a version of Perron-Frobenius theorem \cite{Bir} that there exists a nef class $\alpha \in NS_{\Bbb{R}}(X)$ such that $f^*\alpha=\lambda_1(f)\alpha.$ Now, as $f^*$ preserves the intersection product and $\lambda_1(f)^{k-r-1}>\lambda_{k-r-1}(f)$ we see that $\alpha^{k-r-1}=0$ in $H^{k-r-1,k-r-1}(X,\Bbb{R})$. Therefore, $\nu(\alpha)\leq k-r-2.$ Thus, the assertion follows from the next lemma:
\begin{lem}
Let $X=X_m$ be a rational manifold where $\pi_i:X_{i+1}\to X_i$ is obtained by blowing up a smooth irreducible subvariety $Y_i \subset X_i$ of dimension at most $r$ and $X_0=\pk$ with $k\geq r+2.$ If $\alpha \in \nef$ is non-zero then $\nu(\alpha)\geq k-r-1.$
\end{lem}
\begin{proof}
 It follows from \cite{Kawamata} that 
 $$\nu(\alpha)=\max\{p:\alpha^p\cdot A^{k-p}\neq 0\}$$
 where $A$ is any ample divisor. Therefore, it is enough to show that there exists a divisor $D$ such that $\alpha^{k-r-1}\cdot D^{r+1}\neq0.$ \\ \indent
 It is classical that \cite{GH} the Picard group $Pic(X)$ is generated by the classes $ H_X,E_1,\dots,E_m$ where $$\pi=\pi_{m-1}\circ\pi_{m-2}\circ \dots \circ \pi_1:X\to X_0=\pk $$
 $$H_X:=\pi^*(H)$$  $H$ is the class of a generic hyperplane in  $\pk$ and $E_i$ is the exceptional divisor of the blow up $\pi_i:X_{i+1}\to X_i$ and $$E_{i-1}:= \overline{\pi_{i}^{-1}(E_{i-1}-Y_i)}$$ is the class of the proper transform in $X_i$ of the exceptional divisor $E_{i-1}.$ Then we can represent the class $\alpha$ as $$\alpha=aH_X+\sum_ic_iE_i.$$ where $a,c_i\in\Bbb{R}$. Since $\pi(E_i) \subset \pk$  has codimension at least 2, a generic line in $\pk$ does not intersect $\pi(E_i).$ Then by the projection formula \cite{Fulton} we have $$E_i\cdot H_X^{k-1}=0.$$ Since $\alpha$ is nef and nonzero this implies that $\alpha \cdot H_X^{k-1}=a>0.$ \\ \indent
 Now, since the dimension of $Y_i$ is at most $r,$ a generic subvariety of $\pk$ of codimension $r+1$ does not intersect $\pi(E_i).$ This in turn implies that $E_i^{k-r-1}\cdot H_X^{r+1}=0$ hence, $$\alpha^{k-r-1}\cdot H_X^{r+1}=a^{k-r-1}>0.$$   
\end{proof}
 Hence, we deduce that $\alpha^j\neq0$ in $H^{j,j}(X,\Bbb{R})$ and $\lambda_1(f)^j=\lambda_j(f)$ for all $1\leq j\leq k-r-1.$ Therefore, applying the same argument to $f^{-1}$ and using Lemma \ref{degree} we conclude that
\begin{eqnarray*}
\lambda_1(f)^{(k-r-1)^2}&=&\lambda_{k-r-1}(f)^{k-r-1}=\lambda_{r+1}(f^{-1})^{k-r-1}=\lambda_1(f^{-1})^{(k-r-1)(r+1)}\\
&=&\lambda_{k-r-1}(f^{-1})^{r+1}=\lambda_{r+1}(f)^{r+1}= \lambda_1(f)^{(r+1)^2}
\end{eqnarray*}
since $k>2+2r$ this contradicts $\lambda_1(f)>1.$
\end{proof}

\begin{proof}[Proof of Theorem \ref{numerical}]
 Let $f\in Aut(X)$ and assume that the first dynamical degree, $\lambda_1>1$  we will derive a contradiction. Since $f^*$ preserves the nef cone there exists a class $\beta \in \nef$ such that $f^*\beta=\lambda_1 \beta$.  \\ \indent
  Now, $\nu(\alpha)\geq k-1$ implies that $\alpha^{k-2}\not=0$ in $H^{k-2,k-2}(X,\Bbb{R}).$ On the other hand, since $f$ is an automorphism it preserves the cup product hence $$f^*(\alpha^{k-1}\wedge \beta)=\lambda_1 \alpha^{k-1}\wedge \beta.$$ Since the topological degree  $\lambda_{k}=1$ we must have $$\alpha^{k-1}\wedge \beta=0.$$ Then, by Proposition \ref{DS} there exists (up to a scaler multiple) unique real numbers $(a,b)\not=(0,0)$  such that $$(a\beta+b\alpha)\wedge \alpha^{k-2}=0.$$ Pulling-back this equation by $f,$ we obtain 
  $$(a\lambda_1\beta+b\alpha)\wedge \alpha^{k-2}=0. $$  
  Since $\lambda_1>1$, we see that $b=0.$ Thus, $$\beta \wedge \alpha^{k-2}=0.$$ Applying the same argument repeatedly we obtain that $$\beta\wedge \alpha=0.$$ Then Proposition \ref{DS} implies that $\beta=c\alpha$ for some $c\in \Bbb{R}_+$ but this contradicts $\lambda_1>1.$
 \end{proof}
 The following result is an immediate corollary of Theorem \ref{numerical} and \cite[Proposition 2.2]{Kawamata}:

\begin{cor}\label{itaka}
Let $X$ be a projective manifold and $f\in Aut(X).$ If there exists a nef $\Bbb{R}$-divisor $L$ such that $\kappa(L)\geq k-1$ and $f^*L\cong L$ then $h_{top}(f)=0.$ 
\end{cor}
 
Recall that a compact complex manifold is called \textit{Fano} if the anti-canonical bundle $-K_X$ is ample.  It follows from Kodaira embedding theorem that a Fano manifold is projective. 
More generally, a compact complex manifold is called \textit{weak Fano} if the anti-canonical bundle $-K_X$ is big and nef. The following immediate corollary is well-known \cite{Zhang}:

\begin{cor}\label{fano}
Let $X$ be a projective weak Fano manifold and $f\in Aut(X)$ then $h_{top}(f)=0$.
 \end{cor}
 \begin{proof}
Note the $f$ preserves the divisor class $-K_X$ which is big and nef by definition of $X.$ Thus, the assertion follows from the Corollary \ref{itaka}.
\end{proof} 
 Blanc and Lamy \cite{BL} recently proved that blow up of the complex projective space $\Bbb{P}^3$ along a curve $C$ of degree $d$ and genus $g$ lying on a smooth quadric  gives a weak Fano manifold if $4d-30\leq g \leq 14$ or $(g,d)=(19,12)$ or there is no 5-secant line, 9-secant conic, nor 13-secant twisted cubic to $C.$ In particular, Corollary \ref{fano} implies that blowing up $\Bbb{P}^3$ along such curves does not give rise to a rational manifold admitting an automorphism with positive entropy. More generally, it was observed in \cite{tuyen} that if a rational manifold is obtained from $\Bbb{P}^3$ by blowing up finitely many curves with no common intersection then any automorphism of $X$ is of zero entropy.

 \bibliographystyle{alpha}
\bibliography{biblio}
\end{document}